\documentclass[12pt,leqno]{article}
\usepackage{graphicx} 
\usepackage{amsfonts}
\usepackage{graphicx, amsmath, amsthm, amsfonts, amssymb, color, ulem}
\usepackage{cite}
\usepackage{mathrsfs, float}
\newtheorem{thm}{Theorem}[section]
\newtheorem{ass}[thm]{Assumption}

\newtheorem{lem}[thm]{Lemma}

\newtheorem{exa}[thm]{Example}
\newtheorem{rem}[thm]{Remark}
\theoremstyle{definition}

\definecolor{wco}{rgb}{0.5,0.2,0.3}
\renewcommand{\bar}{\overline}

\numberwithin{equation}{section}

\title{{\bf Convergence rate of nonlinear delayed neutral McKean-Vlasov SDEs driven by fractional Brownian motions\thanks{Supported by Natural Science Foundation of China(12061034, 12271524), Natural Science Foundation of Jiangxi(20224ACB201002), Natural Science Foundation of Hunan(2022JJ30674), China Postdoctoral Science Foundation(2022M711424), and Special Funding Project for Graduate Innovation of Jiangxi(YC2023-B182). }}
}

\author
{ {\bf Shengrong Wang}$^{\tt a, b}$, {\bf Jie Xie}$^{\tt a}$, \
{\bf Li Tan}$^{\tt a, c}$ \thanks{Contact E-mail address: tltanli@126.com}
\\[0.5ex]
$^{\tt a}$School of Statistics and Data Science, Jiangxi University of \\ Finance and Economics, Nanchang, Jiangxi, 330013, China \\
$^{\tt b}$College of Modern Economics and Management of Jiangxi University\\ of Finance and Economics, Jiujiang, Jiangxi, 332020, China\\
$^{\tt c}$Key Laboratory of Data Science in Finance and Economics, Jiangxi \\ University of Finance and Economics, Nanchang, Jiangxi, 330013, China}

\begin{document}

\maketitle
\begin{abstract}
 In this paper, our main aim is to investigate the strong convergence for a neutral McKean-Vlasov stochastic differential equation with super-linear delay driven by fractional Brownian motion with Hurst exponent $H\in(1/2, 1)$. After giving uniqueness and existence for the exact solution, we analyze the properties including boundedness of moment and propagation of chaos. Besides, we give the Euler-Maruyama (EM) scheme and show that the numerical solution converges strongly to the exact solution. Furthermore, a corresponding numerical example is given to illustrate the theory. 
\end{abstract}

\section{Introduction}
The McKean-Vlasov stochastic differential equations (SDEs) have received a lot of attention since they were proposed by McKean \cite{MC66}. The concept of chaos propagation is defined in terms of relative entropy. Bossy and Talay \cite{BT97,BT96} chose to approximate the McKean-Vlasov
limit by replicating the behavior with a system of $n$ weakly interacting particles, and studied the numerical approximation of the solutions to the McKean-Vlasov SDEs. Lacker \cite{L18} showed how to considerably strengthen the usual mode of convergence of
an $n$-particle system to its McKean-Vlasov limit, known as propagation of chaos. Hammersley et al. \cite{HSS21} investigated the McKean-Vlasov SDEs with common noise, and developed an appropriate definition of
weak solutions for such equations. The existence and uniqueness of the exact solutions as well as the propagation of chaos of McKean-Vlasov SDEs, have been intensively studied, for example in \cite{DST19,W18} and references therein.

Some studies have considered the situation where the coefficients of McKean-Vlasov SDEs exhibit super-linear growth. Dos Reis et al.\cite{DES22} presented two fully probabilistic Euler schemes for the simulation of McKean-Vlasov SDEs: one explicit and one implicit. These schemes are designed for McKean-Vlasov SDEs under the one-sided Lipschitz drift assumption and the polynomial growth in diffusion. Reisinger and Stockinger \cite{RS22} introduced adaptive EM schemes for McKean-Vlasov SDEs with one-sided Lipschitz condition and Polynomial growth condition for the drift, and proved strong convergence rate of $1/2$. Gao et al. \cite{gghy24} studied the numerical scheme of highly nonlinear neutral multiple-delay stochastic McKean-Vlasov equations, and established the tamed EM
scheme for the corresponding particle system, and obtained the convergence rate in $\mathcal{L}^p$ sense. Cui et al. \cite{clly22} studied numerical methods for approximating solutions to a sort of McKean-Vlasov neutral SDEs, where the drift coefficient exhibits super-linear growth, and constructed the strong convergence rate of the tamed EM scheme.

Generally, let $B^H=\{B_{t}^H, t\in [0,T]\}$ be a fractional Brownian motion with Hurst exponent $H\in (0,1)$ defined on the probability space $(\Omega,\mathcal{F},\mathbb{P})$. Then, $B^H$ is a centered Gaussian process
with the covariance function
\begin{equation*}
 R_{H}(t,s)=\mathbb{E}(B_{t}^{H}B_{s}^{H})=\frac{1}{2}(t^{2H}+s^{2H}-|t-s|^{2H}),\ \ \forall t,s\in[0,T].
\end{equation*}

The fractional Brownian motion$\{B_{t}^H\}_{t\geq 0}$ corresponds to a standard Brownian motion when $H=1/2$. This process was
introduced by Kolmogorov and studied by Mandelbrot and Van Ness \cite{BEV68}. The self-similarity and long memory properties make fractional Brownian motion a suitable input noise for various models. It is said that $B^H$ is neither a Markov process nor a semimartingale unless $H=1/2$. However, it has been shown that many basic tools in ordinary stochastic calculus have their counterparts for fractional Brownian motion\cite{Y19}. He et al. \cite{HGZG23} proposed a truncated EM scheme for SDEs driven by fractional Brownian motion with super-linear drift coefficient and obtained the convergence rate of the numerical method. For more details about SDEs driven by fractional Brownian motion, we can refer to \cite{HP09,LHH23,WDX22,ZY21} and the related references therein.

Recently, fractional Brownian motion has been applied to financial time series, hydrology, and telecommunications. To develop these applications, numerous equations have been established regarding fractional Brownian motion, including McKean-Vlasov SDEs. However, compared to SDEs driven by standard Brownian motion, the numerical issues of SDEs driven by fractional Brownian motion have not been well studied. Among them, the research on McKean-Vlasov SDEs is even scarcer. To name a few, Bauer and Meyer-Brandis\cite{BM19} studied McKean-Vlasov equations on infinite-dimensional Hilbert spaces with irregular drift and additive fractional noise. Shen et al. \cite{SXW22} established the existence and uniqueness theorem for solutions of the McKean-Vlasov SDEs driven by fractional Brownian motion and standard Brownian motion. Fan et al. \cite{FHS22} studied a McKean-Vlasov SDE driven by fractional Brownian motion as follows:
\begin{equation}\label{Hedemoxing}
{\rm d}X_{t}=b(t, X_{t},\mathcal{L}(X_{t})){\rm d}t+\sigma(t, \mathcal{L}(X_{t})){\rm d}B_{t}^H,
\end{equation}
where the coefficients $b:[0,T]\times\mathbb{R}^d\times\mathcal{P}_\theta(\mathbb{R}^d)\to \mathbb{R}^d$, $\sigma:[0,T]\times\mathcal{P}_{\theta}(\mathbb{R}^d)\to \mathbb{R}^d\otimes\mathbb{R}^d$, here, $ \mathcal{P}_{\theta}(\mathbb{R}^d)$ is the set of probability measures on $\mathbb{R}^d$ with finite $\theta$-th moment. They gave the well-posedness of \eqref{Hedemoxing} and established the Bismut formula by Malliavin calculus. He et al.\cite{HGZ24} studied the EM scheme for \eqref{Hedemoxing} where the coefficients are $\mathbb{W}_{\theta}$-Lipschitz.

Since stochastic differential delay equations have been widely discussed in many researches to describe the dependence of past behavior, for example \cite{r8,r11}. Neutral McKean-Vlasov SDDEs have also been discussed, such as \cite{clly22}, \cite{gghy24}. However, there are almost no studies on neutral McKean-Vlasov SDEs with delay driven by fractional Brownian motion, and even fewer that address the case where the coefficients exhibit nonlinear growth as well. Motivated by the above discussion, we consider neutral McKean-Vlasov SDEs with super-linear delay driven by fractional Brownian motion, as an extension of \eqref{Hedemoxing}
\begin{equation*}
 {\rm d}\big(X_{t}-D(X_{t-\tau
 })\big)=b(X_{t},X_{t-\tau},\mathcal{L}(X_{t})){\rm d}t+\sigma(\mathcal{L}(X_{t})){\rm d}B_{t}^H,
\end{equation*}
where $D(\cdot)$ is a neutral term and $\tau>0$ is the delay.

The structure of the rest of the paper is as follows. In Section 2, the mathematical preliminaries on the neutral McKean-Vlasov SDEs with super-linear delays driven by fractional Brownian motion are presented. In Section 3, existence and uniqueness of the solution is given. In Section 4, the propagation of chaos for this model is also analyzed. In Section 5, the EM scheme is established and the strong convergence rate is given. Finally, a numerical example is presented in Section 6.

\section{Preliminaries}
Throughout the paper we work on a filtered probability space $\{\Omega,\mathcal{F},\{\mathcal{F}\}_{t\geq 0},\mathbb{P}\}$ satisfying the usual
conditions. The canonical process $\{B^H_{t}\}_{t\in[0,T]}$ is a $d$-dimensional fractional Brownian motion defined on the probability space with Hurst exponent $H\in(0,1)$. Let $(\mathbb{R}^d,\langle\cdot,\cdot\rangle,|\cdot|)$ be the $d$-dimensional Euclidean space with inner product $\langle\cdot,\cdot\rangle$ and Euclidean norm $|\cdot|$. For a matrix, we denote by $\|\cdot\|$ the Euclidean norm. In addition, we use $\mathcal{P}(\mathbb{R}^d)$ to denote the family of all probability measures on $(\mathbb{R}^d,\mathcal{B}(\mathbb{R}^d))$, where $\mathcal{B}(\mathbb{R}^d)$ denotes the Borel $\sigma$-field over $\mathbb{R}^d$. For any $\tau>0$, let $\mathcal{C}([-\tau,0];\mathbb{R})$ be the family of all continuous functions $\varphi$ from $[-\tau,0]$ to $\mathbb{R}^d$ with $\|\varphi\|:=\sup\limits_{-\tau\le\rho\le0}|\varphi(\rho)|$. Define the subset of probability measures with finite $\theta$-th moment by
\begin{equation*}
 \mathcal{P}_{\theta}(\mathbb{R}^d)=\bigg\{\mu\in\mathcal{P}(\mathbb{R}^d):\int_{\mathbb{R}^d}|x|^{\theta}\mu({\rm d}x)<\infty\bigg\},
\end{equation*}
and for $\mu\in\mathcal{P}_{\theta}(\mathbb{R}^d)$, define
\begin{equation*}
\mathcal{W}_{\theta}(\mu)=\left(\int_{\mathbb{R}^d}|x|^{\theta}\mu({\rm d}x)\right)^{1/\theta}.
\end{equation*}
As a metric on the space, we use the Wasserstein distance. For $\mu,\nu \in\mathcal{P}_{\theta}(\mathbb{R}^d)$ with $\theta\ge1$, the Wasserstein
distance between $\mu$ and $\nu$ is defined as
\begin{equation*}\label{wassersteindistance}
 \mathbb{W}_{\theta}(\mu,\nu)=\bigg(\inf_{\pi\in\Pi(\mu,\nu)} \int_{\mathbb{R}^d\times\mathbb{R}^d}|x-y|^{\theta}\pi({\rm d}x,{\rm d}y)\bigg)^{1/\theta},
\end{equation*}
where $\Pi(\mu,\nu)$ stands for all the couplings of $\mu$ and $\nu$, i.e., $\pi\in \Pi(\mu, \nu)$  is a probability measure on $\mathbb{R}^d\times\mathbb{R}^d$ such that $\pi(\cdot,\mathbb{R}^d)=\mu(\cdot)$ and $\pi(\mathbb{R}^d,\cdot)=\nu(\cdot)$. For $p\ge 1$, let $L^p(\mathbb{R}^d)$ be the space of $\mathbb{R}^d$-valued random variables $V$ that satisfies $\parallel V(x)\parallel_{L^p}:=(\mathbb{E}|V(x)|^p)^{1/p}<\infty.$

It is worth noting that there are two major obstacles depending on the properties of $\{B^H_{t}\}_{t\in[0,T]}$. Firstly, the fractional Brownian motion
is not a semimartingale except in the case of standard Brownian motion ($H=1/2$), hence the classical It\^{o} calculus based on semimartingales cannot be directly applied to the fractional case. Secondly, there is no martingale representation theorem with respect to the fractional Brownian motion.

In this paper, we consider the following $d$-dimensional neutral McKean-Vlasov SDDE driven by fractional Brownian motion as follows:
\begin{equation}\label{pp}
\left\{
\begin{array}{ll}
     {\rm d}\big(X_{t}-D(X_{t-\tau})\big)=b(X_{t},X_{t-\tau},\mathcal{L}_{X_{t}}){\rm d}t+\sigma(\mathcal{L}_{X_{t}}){\rm d}B^{H}_{t},\ \ t>0,   \\
     X_{t}=\xi, \ \ t\in[-\tau,0] ,
\end{array}
\right.
\end{equation}
where $D:\mathbb{R}^d\to\mathbb{R}^d$, $b:\mathbb{R}^d\times\mathbb{R}^d\times\mathcal{P}_\theta(\mathbb{R}^d)\to \mathbb{R}^d$, $\sigma:\mathcal{P}_{\theta}(\mathbb{R}^d)\to \mathbb{R}^d\otimes\mathbb{R}^d$, $\mathcal{L}_{X_{t}}$ denotes the law of $X_{t}$ and $H\in(1/2,1)$, the initial data $\xi\in L^{\bar{p}}(\Omega\to \mathbb{R}^d,\mathcal{F}_{0},\mathbb{P})$ with $\bar{p}\geq \theta\geq 1$. We impose the following assumptions.

\begin{ass}\label{duoxiangshizengz}
{\rm $D(0)=0$. And there exist positive constants $L$ and $l\geq 1$, $0< \lambda< 1$, such that
\begin{align*}
  |D(x)-D(y)|\leq \lambda|x-y|,
\end{align*}
\begin{equation*}
\begin{array}{ll}
     |b(x,y,\mu)-b(\bar{x},\bar{y},\nu)|\leq L\left[|x-\bar{x}|+(1+|y|^l+|\bar{y}|^l)|y-\bar{y}|+\mathbb{W}_{\theta}(\mu,\nu)\right],
\end{array}
\end{equation*}
and
\begin{equation*}
\begin{array}{ll}
  \parallel\sigma(\mu)-\sigma(\nu)\parallel \leq L\mathbb{W}_{\theta}(\mu,\nu),
\end{array}
\end{equation*}
for all $x,y,\bar{x},\bar{y}\in\mathbb{R}^d$, $\mu,\nu\in\mathcal{P}_{q}(\mathbb{R}^d)$. And for initial experience distribution $\delta_{0}$,
\begin{equation*}
|b(0,0,\delta_0)|\vee\parallel\sigma(\delta_0)\parallel\leq L.
\end{equation*}
}
\end{ass}

\begin{lem}\label{xishuzengzhangtiaojian}
{\rm Let Assumption \ref{duoxiangshizengz} hold. There exists a positive constant $L$ such that
\begin{equation*}
|b(x,y,\mu)|\leq L\left(1+|x|+|y|^{l+1}+\mathbb{W}_{\theta}(\mu,\delta_0)\right),
\end{equation*}
and
\begin{equation*}
\parallel\sigma(\mu)\parallel\leq L\left(1+\mathbb{W}_{\theta}(\mu,\delta_0)\right).
\end{equation*}
}
\end{lem}

\begin{proof}
 We omit it since it is obvious.
\end{proof}

\begin{lem}\cite[Lemma 4.1]{M07}\label{Maodebudengshi}
 {\rm Let $p>1$, $\varepsilon>0$ and $x,y\in\mathbb{R}$. Then}
 \begin{align*}
   |x+y|^p\leq \bigg(1+\varepsilon^{\frac{1}{p-1}}\bigg)^{p-1}\bigg(|x|^p+\frac{|y|^p}{\varepsilon}\bigg).
 \end{align*}
 {\rm Especially for $\varepsilon=((1-\lambda)/\lambda)^{p-1}$, where $p\geq 2$ and $0<\lambda<1$, we get}
 \begin{align*}
  |x+y|^p\leq \frac{1}{\lambda^{p-1}}|x|^p+\frac{1}{(1-\lambda)^{p-1}}|y|^p.
 \end{align*}
\end{lem}

\section{Existence and Uniqueness}
In stead of the usual Picard iteration, we use the Carath\'{e}odory approximation to show the existence and uniqueness for \eqref{pp}. Since for the Picard iteration, one can not be able to establish a Cauchy sequence in case of equations with super-linearly growing coefficients. We first give the definition of Carath\'{e}odory approximate as follows. For every integer $n\geq 2/\tau$, define $X_{t}^n$ on $[-\tau,T]$ by
\begin{align*}
 X_{t}^n=\xi(t),\ \ t\in[-\tau,0],
\end{align*}
and for $t\in[0,T]$,
\begin{equation}\label{Caratheodory  s approximat system}
\begin{aligned}
X_{t}^n-D(X_{t-\tau}^n)=\xi(0)+D(\xi(-\tau))+\int_{0}^tb\left(X^n_{s-1/n},X^n_{s-\tau},\mathcal{L}_{X_{s-1/n}^{n}}\right){\rm d}s
+\int_{0}^t\sigma\left(\mathcal{L}_{X_{s-1/n}^{n}}\right){\rm d}B_{s}^H.
\end{aligned}
\end{equation}
It is important to note that each $X_{t}^n$ can be determined explicitly by the stepwise iterated It\^{o} integrals over the intervals $[0, 1/n]$, $(1/n, 2/n]$ etc. To obtain the desired results, it is also necessary to establish the following lemmas.

\begin{lem}\label{boundedness of pp}
{\rm Let Assumption \ref{duoxiangshizengz} hold. Then for $\bar{p}>1/H$ and $n\geq2/\tau$, we have
\begin{equation*}
 \mathbb{E}\bigg(\sup_{-\tau\leq t\leq T}|X_{t}^n|^{\bar{p}} \bigg)\leq C,
\end{equation*}
where $C$ is a positive constant depending on $\bar{p}, H, \xi, T, L$.
}
\end{lem}

\begin{proof}
For any $t\in[0,T]$, by Assumption \ref{duoxiangshizengz} and Lemma \ref{Maodebudengshi}, we immediately get
\begin{align*}
\lvert X_{t}^n\rvert^{\bar{p}}=&\lvert D(X_{t-\tau}^n)+X_{t}^n-D(X_{t-\tau}^n) \rvert^{\bar{p}}\\
 \leq &\frac{1}{\lambda^{\bar{p}-1}}|D(X_{t-\tau}^n)|^{\bar{p}}+\frac{1}{(1-\lambda)^{\bar{p}-1}}|X_{t}^n-D(X_{t-\tau}^n)|^{\bar{p}}\\
 \leq &\lambda|X_{t-\tau}|^{\bar{p}}+\frac{1}{(1-\lambda)^{\bar{p}-1}}|X_{t}^n-D(X_{t-\tau}^n)|^{\bar{p}},
\end{align*}
then,
\begin{align*}
  \sup_{0\leq s\leq t}|X_{s}^n|^{\bar{p}}\leq \lambda\parallel\xi\parallel^{\bar{p}}+\lambda \sup_{0\leq s\leq t}|X_{s}^n|^{\bar{p}}+\frac{1}{(1-\lambda)^{\bar{p}-1}}\sup_{0\leq s\leq t}|X_{s}^n-D(X_{s-\tau}^n)|^{\bar{p}}.
\end{align*}
Therefore, we get
\begin{align*}
  \sup_{0\leq s\leq t}|X_{s}^n|^{\bar{p}}\leq \frac{\lambda}{1-\lambda}\parallel\xi\parallel^{\bar{p}}+\frac{1}{(1-\lambda)^{\bar{p}}}\sup_{0\leq s\leq t}|X_{s}^n-D(X_{s-\tau}^n)|^{\bar{p}}.
\end{align*}
By the elementary inequality, the H\"{o}lder inequality and Lemma \ref{xishuzengzhangtiaojian}, we derive from \eqref{Caratheodory  s approximat system} that
\begin{align*}
&\mathbb{E}\bigg(\sup_{s\in[0,t]}|X_{s}^{n}-D(X_{s-\tau}^n)|^{\bar{p}}\bigg)\\
\leq& 3^{\bar{p}-1}\mathbb{E}|\xi(0)-D(\xi(-\tau))|^{\bar{p}}+3^{\bar{p}-1}\mathbb{E}\bigg(\sup_{s\in[0,t]}\bigg\lvert \int_{0}^sb\left(X_{r-1/n}^n,X_{r-\tau}^n,\mathcal{L}_{X_{r-1/n}^n}\right){\rm d}r\bigg\rvert^{\bar{p}}\bigg)\\
&+3^{\bar{p}-1}\mathbb{E}\bigg(\sup_{s\in[0,t]}\bigg\lvert \int_{0}^s\sigma\left(\mathcal{L}_{X_{r-1/n}^n}\right){\rm d}B_{r}^H\bigg\rvert^{\bar{p}}\bigg)\\
\leq &C+C\mathbb{E}\int_{0}^{t}\left[1+|X_{r-1/n}^n|^{\bar{p}}+|X_{r-\tau}^n|^{(l+1){\bar{p}}}+\left(\mathbb{W}_{\theta}(\mathcal{L}_{X_{r-1/n}^n},\delta_0)\right)^{\bar{p}}\right]{\rm d}r\\
&+C\mathbb{E}\bigg(\sup_{s\in[0,t]}\bigg\lvert \int_{0}^s\sigma(\mathcal{L}_{X_{r-1/n}^n}){\rm d}B_{r}^H\bigg\rvert^{\bar{p}}\bigg).
\end{align*}
For the last term, we use the techniques of Theorem 3.1 in \cite{FHS22}. Denote $\int_s^t(t-\kappa)^{-\kappa}(r-s)^{\kappa-1}{\rm d}r:=C(\kappa)$. Since for $\bar{p}H>1$, we can take some $\kappa$ such that $1-H<\kappa<1-1/{\bar{p}}$, then by the stochastic Fubini theorem, the H\"{o}lder inequality, and combining with Lemma \ref{xishuzengzhangtiaojian}, we get
\begin{equation}\label{2.14}
\begin{aligned}
&\mathbb{E}\bigg(\sup_{s\in[0,t]}\bigg\lvert \int_{0}^s\sigma(\mathcal{L}_{X_{r-1/n}^n}){\rm d}B_{r}^H\bigg\rvert^{\bar{p}}\bigg)\\
\leq& \frac{C(\kappa)^{-{\bar{p}}}}{({\bar{p}}-1-\kappa {\bar{p}})^{{\bar{p}}-1}}t^{{\bar{p}}-1-\kappa {\bar{p}}}\int_{0}^{t}\mathbb{E}\bigg|\int_{0}^s(s-r)^{\kappa-1}\sigma(\mathcal{L}_{X_{r-1/n}^n}){\rm d}B_{r}^H\bigg|^{{\bar{p}}}{\rm d}s\\
\leq& Ct^{{\bar{p}}-1-\kappa{\bar{p}}}\int_{0}^{t}\left(\int_0^s(s-r)^{\frac{\kappa-1}{H}}\parallel\sigma(\mathcal{L}_{X_{r-1/n}^n})\parallel^{\frac{1}{H}}{\rm d}r\right)^{\bar{p}H}{\rm d}s\\
\leq &Ct^{\bar{p}H-1}\int_{0}^{t}\parallel\sigma(\mathcal{L}_{X_{r-1/n}^n})\parallel^{\bar{p}}{\rm d}r\\
\leq& Ct^{\bar{p}H-1}L^{\bar{p}}\int_{0}^{t}\left[1+\mathbb{W}_\theta(\mathcal{L}_{X_{r-1/n}^n},\delta_0)\right]^{\bar{p}}{\rm d}r.
\end{aligned}
\end{equation}
Referring to the proof in \cite[Proposition 3.4]{DST19}, for the $\theta$-Wasserstein metric, we get
\begin{align*}
\left(\mathbb{W}_{\theta} (\mathcal{L}_{X_{r-1/n}^n},\delta_{0})\right)^{\bar{p}}=\left(\mathcal{W}_{\theta}(\mathcal{L}_{X_{r-1/n}^n})\right)^{\bar{p}}\leq \mathbb{E}|X_{r-1/n}^n|^{\bar{p}}.
\end{align*}
To sum up, by the Gronwall inequality, we get
\begin{equation}\label{20240510}
\mathbb{E}\bigg(\sup_{s\in[0,t]}|X_{s}^n|^{\bar{p}}\bigg)\leq C+C\int_{0}^t\mathbb{E}|X_{r-\tau }^n|^{(l+1)\bar{p}}{\rm d}r.
\end{equation}
Taking the skills of \cite[Lemma 2.1]{by13} as a reference, to get the assertion, we first set up a finite sequence $\{u_{1},u_{2},\cdots,u_{[T/\tau]+1}\}$ where $u_i=([T/\tau]+2-i)\bar{p}(l+1)^{[T/\tau]+1-i}$, and $[T/\tau]$ denotes the integer part of $T/\tau$. It is easy to see that $u_{i}\ge \bar{p}$, $u_{[T/\tau]+1}=\bar{p}$ and $(l+1)u_{i+1}<u_{i}$ for $i=1,2,\cdots,[T/\tau]$.
For $s\in[0,\tau]$, we immediately get
\begin{align*}
\mathbb{E}\bigg(\sup_{s\in[0,\tau]}|X_{s}^n|^{u_1}\bigg)\leq C
\end{align*}
since $\xi\in L^{\bar{p}}(\Omega\to \mathbb{R}^d,\mathcal{F}_{0},\mathbb{P})$. For $s\in[0,2\tau]$, by \eqref{20240510} and the H\"{o}lder inequality, we get
\begin{align*}
&\mathbb{E}\bigg(\underset{s\in[0,2\tau]}{\sup}|X_{s}^{n}|^{u_{2}}\bigg)
\leq C+C \int_{0}^{2\tau}\left[\mathbb{E}\bigg(\sup_{s\in[0,r]}|X_{s-\tau}^n|^{u_{1}}\bigg)\right]^{\frac{(l+1)u_{2}}{u_{1}}}{\rm d}r\leq C.
\end{align*}
By induction, the desired result is obtained.
\end{proof}

\begin{lem}\label{jiedechazhijuyoujie}
 {\rm Let Assumption \ref{duoxiangshizengz} hold. Then for $0\leq s\le t\leq T $ with $t-s<1$, and $p>1/H$, $(l+1)p\le\bar{p}$, $n\geq 2/\tau$, we have
 \begin{align*}
\mathbb{E}|X_{t}^n-X_{s}^n|^p\leq C[(t-s)^{p}+(t-s)^{pH}],
 \end{align*}
where $C$ is a positive constant depending on $p, H, \xi, T, L$ but independent of $s$ and $t$.
}
\end{lem}

\begin{proof}
Similarly to Lemma \ref{boundedness of pp}, we get
\begin{align*}
  |X_{t}^n-X_{s}^n|^p\leq \frac{1}{(1-\lambda)^p}|X_{t}^n-D(X_{t-\tau}^n)-X_{s}^n+D(X_{s-\tau}^n)|^p.
\end{align*}
By \eqref{Caratheodory  s approximat system}, we may compute
\begin{align*}
X_{t}^n-D(X_{t-\tau}^n)-X_{s}^n+D(X_{s-\tau}^n)=&\int_{s}^tb\bigg(X^n_{r-1/n},X^n_{r-\tau},\mathcal{L}_{X_{r-1/n}^{n}}\bigg){\rm d}r+\int_{s}^t\sigma(\mathcal{L}_{X_{r-1/n}^{n}}){\rm d}B_{r}^H.
\end{align*}
By the elementary inequality, the H\"{o}lder inequality, Lemma \ref{xishuzengzhangtiaojian} and similar skills as \eqref{2.14}, we get
\begin{align*}
&\mathbb{E}|X_{t}^n-D(X_{t-\tau}^n)-X_{s}^n+D(X_{s-\tau}^n)|^p\\
\leq &2^{p-1}\mathbb{E}\bigg\lvert \int_{s}^tb\left(X_{r-1/n}^n,X_{r-\tau}^n,\mathcal{L}_{X_{r-1/n}^n}\right){\rm d}r\bigg\rvert^{p}+2^{p-1}\mathbb{E}\bigg\lvert \int_{s}^t\sigma(\mathcal{L}_{X_{r-1/n}^n}){\rm d}B_{r}^H\bigg\rvert^{p}\\
\leq & 2^{p-1}(t-s)^{p-1}\mathbb{E}\int_{s}^{t}L^{p}\left[1+|X_{r-1/n}^n|+|X_{r-\tau}^n|^{l+1}+\mathbb{W}_{\theta}(\mathcal{L}_{X_{r-1/n}^n},\delta_0)\right]^{p}{\rm d}r\\
&+2^{p-1}C(t-s)^{pH-1}\int_{s}^tL^p\left[1+\mathbb{W}_{\theta}(\mathcal{L}_{X_{r-1/n}^n},\delta_0)\right]^p{\rm d}r\\
\leq &C(t-s)^{p-1}\int_{s}^t\mathbb{E}\bigg(\sup_{u\in[-\tau,r]}|X_{u}^n|^p\bigg){\rm d}r+C(t-s)^{p-1}\int_{s}^t\mathbb{E}\bigg(\sup_{u\in[-\tau,r]}|X_{u}^n|^{(l+1)p}\bigg){\rm d}r\\
&+C(t-s)^{pH-1}\int_{s}^t\mathbb{E}\bigg(\sup_{u\in[-\tau,r]}|X_{u}^n|^p\bigg){\rm d}r.
\end{align*}
Since $(l+1)p\le\bar{p}$, we can applying Lemma \ref{boundedness of pp} to obtain the assertion.
\end{proof}

\begin{thm}\label{existenceanduniquenessofsolution}
{\rm  Let Assumption \ref{duoxiangshizengz} hold. Then for $p>1/H$ and $(l+1)p\le\bar{p}$, \eqref{pp} admits a global unique solution.}
\end{thm}

\begin{proof}
We split the proof into two steps.\\
{\bf Step 1. Existence.} For $m>n\geq 2/\tau$ and $t\in[0,T]$, by Lemma \ref{Maodebudengshi}, we get
\begin{align*}
|X_{s}^m-X_{s}^n|^p\leq \frac{1}{(1-\lambda)^p}|X_{s}^m-D(X_{s-\tau}^m)-X_{s}^n+D(X_{s-\tau}^n)|^p.
\end{align*}
By \eqref{Caratheodory  s approximat system}, we get
\begin{align*}
&\mathbb{E}\bigg(\sup_{s\in[0,t]}|X^m_{s}-D(X_{s-\tau}^m)-X_{s}^n+D(X_{s-\tau}^n)|^p\bigg)\\
\leq& 2^{p-1}\mathbb{E}\bigg(\sup_{s\in[0,t]}\bigg\lvert \int_{0}^s\bigg[b\left(X^m_{r-1/m},X^m_{r-\tau},\mathcal{L}_{X_{r-1/m}^{m}}\right)-b\left(X^n_{r-1/n},X^n_{r-\tau},\mathcal{L}_{X_{r-1/n}^{n}}\right)\bigg]{\rm d}r\bigg\lvert^p\bigg)\\
&+2^{p-1}\mathbb{E}\bigg[\sup_{s\in[0,t]}\bigg\lvert \int_{0}^s\bigg(\sigma(\mathcal{L}_{X_{r-1/m}^{m}})-\sigma(\mathcal{L}_{X_{r-1/n}^{n}})\bigg){\rm d}B_{r}^H\bigg\rvert^p\bigg]\\
=&:I_{1}+I_{2}.
\end{align*}
By the H\"{o}lder inequality and Assumption \ref{duoxiangshizengz}, we get
\begin{align*}
 I_{1}=&2^{p-1}\mathbb{E}\bigg(\sup_{s\in[0,t]}\bigg\lvert \int_{0}^s\bigg[b\left(X^m_{r-1/m},X^m_{r-\tau},\mathcal{L}_{X_{r-1/m}^{m}}\right)-b\left(X^n_{r-1/n},X^n_{r-\tau},\mathcal{L}_{X_{r-1/n}^{n}}\right)\bigg]{\rm d}r\bigg\lvert^p\bigg)\\
 \leq &(2t)^{p-1}\mathbb{E}\int_{0}^t\left|b\left(X^m_{r-1/m},X^m_{r-\tau},\mathcal{L}_{X_{r-1/m}^{m}}\right)-b\left(X^n_{r-1/n},X^n_{r-\tau},\mathcal{L}_{X_{r-1/n}^{n}}\right)\right|^p{\rm d}r\\
 \leq &C\mathbb{E}\int_{0}^t|X^m_{r-1/m}-X^n_{r-1/n}|^p{\rm d}r+C\mathbb{E}\int_{0}^t\left[(1+|X^m_{r-\tau}|^l+|X^n_{r-\tau}|^l)|X^m_{r-\tau}-X^n_{r-\tau}|\right]^p{\rm d}r\\
  \leq &C\mathbb{E}\int_{0}^t|X^m_{r-1/m}-X^m_{r-1/n}|^p{\rm d}r+C\mathbb{E}\int_{0}^t|X^m_{r-1/n}-X^n_{r-1/n}|^p{\rm d}r\\
  &+C\mathbb{E}\int_{0}^t(|1+|X^m_{r-\tau}|^l+|X^n_{r-\tau}|^l)^{p}|X^m_{r-\tau}-X^n_{r-\tau}|^{p}{\rm d}r.
\end{align*}
Since for $l\ge1$, we have $2pl\le (l+1)p\le\bar{p}$, then by the H\"{o}lder inequality, together with Lemma \ref{boundedness of pp} and Lemma \ref{jiedechazhijuyoujie}, we have
\begin{align*}
 I_{1} \leq &C\mathbb{E}\int_{0}^t|X^m_{r-1/m}-X^n_{r-1/n}|^p{\rm d}r+C\mathbb{E}\int_{0}^t|X^m_{r-1/n}-X^n_{r-1/n}|^p{\rm d}r\\
 &+C\int_{0}^t\left[\mathbb{E}(1+|X^m_{r-\tau}|^l+|X^n_{r-\tau}|^l)^{2p}\right]^{\frac{1}{2}}\left(\mathbb{E}|X^m_{r-\tau}-X^n_{r-\tau}|^{2p}\right)^{\frac{1}{2}}{\rm d}r\\
 \leq &C\left[\left(\frac{1}{n}-\frac{1}{m}\right)^{p}+\left(\frac{1}{n}-\frac{1}{m}\right)^{pH}\right]+C\mathbb{E}\int_{0}^t|X^m_{r-1/n}-X^n_{r-1/n}|^p{\rm d}r\\
 &+C\int_{0}^t\left(\mathbb{E}|X^m_{r-\tau}-X^n_{r-\tau}|^{2p}\right)^{\frac{1}{2}}{\rm d}r.
\end{align*}
By Assumption \ref{duoxiangshizengz}, Lemma \ref{boundedness of pp}, Lemma \ref{jiedechazhijuyoujie} and the same techniques as \eqref{2.14}, we get
\begin{align*}
I_{2}=&2^{p-1}\mathbb{E}\bigg[\sup_{s\in[0,t]}\bigg\lvert \int_{0}^s\bigg(\sigma(\mathcal{L}_{X_{r-1/m}^{m}})-\sigma(\mathcal{L}_{X_{r-1/n}^{n}})\bigg){\rm d}B_{r}^H\bigg\rvert^p\bigg]\\
\leq & C\int_{0}^t\parallel\sigma(\mathcal{L}_{X_{r-1/m}^{m}})-\sigma(\mathcal{L}_{X_{r-1/n}^{n}})\parallel^p{\rm d}r\\
\leq &C\int_{0}^tL^p(\mathbb{W}_{\theta}(\mathcal{L}_{X_{r-1/m}^{m}}-\mathcal{L}_{X_{r-1/n}^{n}}))^p{\rm d}r\\
\leq &C\int_{0}^t\mathbb{E}|X_{r-1/m}^{m}-X_{r-1/n}^{n}|^p{\rm d}r\\
\leq&C\int_{0}^t\mathbb{E}|X_{r-1/m}^{m}-X_{r-1/n}^{m}|^p{\rm d}r+C\int_{0}^t\mathbb{E}|X_{r-1/n}^{m}-X_{r-1/n}^{n}|^p{\rm d}r\\
\leq &C\left[\left(\frac{1}{n}-\frac{1}{m}\right)^{p}+\left(\frac{1}{n}-\frac{1}{m}\right)^{pH}\right]+C\int_{0}^t\mathbb{E}|X_{r-1/n}^{m}-X_{r-1/n}^{n}|^p{\rm d}r.
\end{align*}
Combining $I_{1},I_{2}$ and using the Gronwall inequality, we obtain
\begin{align*}
&\mathbb{E}\bigg(\sup_{s\in[0,t]}|X^m_{s}-X_{s}^n|^p\bigg)\\
\leq &C\left[\left(\frac{1}{n}-\frac{1}{m}\right)^{p}+\left(\frac{1}{n}-\frac{1}{m}\right)^{pH}\right]+C\int_{0}^t\left(\mathbb{E}|X^m_{r-\tau}-X^n_{r-\tau}|^{2p}\right)^{\frac{1}{2}}{\rm d}r.
\end{align*}
Define a sequence $\{u_{1},u_{2},\cdots,u_{[T/\tau]+1}\}$ where $u_i=([T/\tau]+2-i)p2^{[T/\tau]+1-i}$, and $[T/\tau]$ denotes the integer part of $T/\tau$. It is easy to see that $u_{i}\ge p$, $u_{[T/\tau]+1}=p$ and $2u_{i+1}<u_{i}$ for $i=1,2,\cdots,[T/\tau]$.
For $s\in[0,\tau]$, we see
\begin{equation}\label{hao}
\mathbb{E}\bigg(\sup_{s\in[0,\tau]}|X^m_{s}-X_{s}^n|^{u_1}\bigg)\to 0, \ {\rm for}\  m,n\to \infty.
\end{equation}
For $s\in[0,2\tau]$, by the H\"{o}lder inequality, we get
\begin{align*}
&\mathbb{E}\bigg(\underset{s\in[0,2\tau]}{\sup}|X^m_{s}-X_{s}^n|^{u_{2}}\bigg)\\
\leq& C\left[\left(\frac{1}{n}-\frac{1}{m}\right)^{p}+\left(\frac{1}{n}-\frac{1}{m}\right)^{pH}\right]+C \int_{0}^{2\tau}\left[\mathbb{E}\bigg(\sup_{s\in[0,r]}|X_{s-\tau}^m-X_{s-\tau}^n|^{u_{1}}\bigg)\right]^{\frac{u_{2}}{u_{1}}}{\rm }dr.
\end{align*}
We derive from \eqref{hao} that
\begin{equation*}
\mathbb{E}\bigg(\underset{s\in[0,2\tau]}{\sup}|X^m_{s}-X_{s}^n|^{u_{2}}\bigg)\to 0, \ {\rm for}\  m,n\to \infty.
\end{equation*}
By induction, we can see that for $t\in[0, T]$,
\begin{equation}\label{Cauchyxulie}
\mathbb{E}\bigg(\sup_{s\in[0,t]}|X^m_{s}-X_{s}^n|^p\bigg)\to 0, \ {\rm for}\  m,n\to \infty.
\end{equation}
Consequently, $(X^n)_{n\geq 2/\tau}$ is a Cauchy sequence for $p>1/H$. Denote the limit of $(X^n)$ by $X$. Therefore, letting $m\to \infty$ in \eqref{Cauchyxulie}, we conclude
\begin{equation}\label{Cauchyxulie111}
 \lim_{n\to \infty} \mathbb{E}\bigg(\sup_{s\in[0,t]}|X_{s}-X_{s}^n|^p\bigg)=0.
\end{equation}
Next, we need to show that the $X_{t}$ is a solution to \eqref{pp}. For $t\in[0,T]$, we get
\begin{align*}
 \mathbb{E}|X_{t}-X^n_{t-1/n}|^p=&\mathbb{E}|X_{t}-X^n_{t}+X^n_{t}-X^n_{t-1/n}|^p\\
 \leq & 2^{p-1}\mathbb{E}|X_{t}-X^n_{t}|^p+2^{p-1}\mathbb{E}|X^n_{t}-X^n_{t-1/n}|^p.
\end{align*}
By Lemma \ref{jiedechazhijuyoujie} and \eqref{Cauchyxulie111}, we have
\begin{align*}
\mathbb{E}|X_{t}-X^n_{t-1/n}|^p\to 0,\ \ {\rm as}\ \ n\to \infty.
\end{align*}
By taking $n\to \infty$ in \eqref{Caratheodory  s approximat system}, we obtain
\begin{align*}
X_{t}-D(X_{t-\tau})=\xi(0)+D(\xi(-\tau))+\int_{0}^tb(X_{s},X_{s-\tau},\mathcal{L}_{X_{s}}){\rm d}s+\int_{0}^t\sigma(\mathcal{L}_{X_{s}}){\rm d}B_{s}^H,\ \ t\in[0,T],
\end{align*}
this indicates $X_{t}$ is a solution to \eqref{pp}. \\
{\bf Step 2. Uniqueness.} Let $X_{t}$ and $Y_{t}$ be two solutions of \eqref{pp} with $X_{t}=Y_{t}=\xi(t),\ \ t\in[-\tau,0]$. For $t\in[0,T]$, by the elementary inequality, we get
\begin{align*}
 &\mathbb{E}\bigg(\sup_{s\in[0,t]}|X_{s}-D(X_{s-\tau})-Y_{s}+D(Y_{s-\tau})|^p\bigg)\\
 \leq &2^{p-1}\mathbb{E}\bigg(\sup_{s\in[0,t]}\bigg\lvert \int_{0}^s(b(X_{r},X_{r-\tau},\mathcal{L}_{X_{r}})-b(Y_{r},Y_{r-\tau},\mathcal{L}_{Y_{r}})){\rm d}r\bigg\rvert^p\bigg)\\
 &+2^{p-1}\mathbb{E}\bigg(\sup_{s\in[0,t]}\bigg\lvert\int_{0}^s\sigma(\mathcal{L}_{X_{r}})-\sigma(\mathcal{L}_{Y_{r}}){\rm d}B_{r}^H\bigg\rvert^p\bigg)\\
 =&W_{1}+W_{2}.
\end{align*}
By Assumption \ref{duoxiangshizengz}, the H\"{o}lder inequality and Lemma \ref{boundedness of pp}, we get
\begin{align*}
 W_{1}=&2^{p-1}\mathbb{E}\bigg(\sup_{s\in[0,t]}\bigg\lvert \int_{0}^s(b(X_{r},X_{r-\tau},\mathcal{L}_{X_{r}})-b(Y_{r},Y_{r-\tau},\mathcal{L}_{Y_{r}})){\rm d}r\bigg\rvert^p\bigg)\\
 \leq &(2t)^{p-1}\mathbb{E}\int_{0}^t|b(X_{r},X_{r-\tau},\mathcal{L}_{X_{r}})-b(Y_{r},Y_{r-\tau},\mathcal{L}_{Y_{r}})|^p{\rm d}r\\
 \leq &C\mathbb{E}\int_{0}^t|X_{r}-Y_{r}|^p{\rm d}r+C\mathbb{E}\int_{0}^t\left[\mathbb{E}(1+|X_{r-\tau}|^{l}+|Y_{r-\tau}|^{l})^{2p}\right]^{\frac{1}{2}}\left(\mathbb{E}|X^m_{r-\tau}-X^n_{r-\tau}|^{2p}\right)^{\frac{1}{2}}{\rm d}r\\
  \leq &C\mathbb{E}\int_{0}^t|X_{r}-Y_{r}|^p{\rm d}r+C\mathbb{E}\int_{0}^t\left(\mathbb{E}|X^m_{r-\tau}-X^n_{r-\tau}|^{2p}\right)^{\frac{1}{2}}{\rm d}r.
\end{align*}
By Assumption \ref{duoxiangshizengz} again, we get
\begin{align*}
W_{2}=2^{p-1}\mathbb{E}\bigg(\sup_{s\in[0,t]}\bigg\lvert\int_{0}^s\sigma(\mathcal{L}_{X_{r}})-\sigma(\mathcal{L}_{Y_{r}}){\rm d}B_{r}^H\bigg\rvert^p\bigg)\leq C\mathbb{E}\int_{0}^t|X_{r}-Y_{r}|^p{\rm d}r.
\end{align*}
Recalling that
\begin{align*}
|X_{s}-Y_{s}|^p\leq \frac{1}{(1-\lambda)^p}|X_{s}-D(X_{s-\tau})-Y_{s}+D(Y_{s-\tau})|^p.
\end{align*}
Combining $W_{1}$-$W_{2}$ and using the Grownall inequality, we obtain
\begin{align*}
&\mathbb{E}\bigg(\sup_{s\in[0,t]}|X_{s}-Y_{s}|^p\bigg)\le C\int_{0}^t\left(\mathbb{E}|X^m_{r-\tau}-X^n_{r-\tau}|^{2p}\right)^{\frac{1}{2}}{\rm d}r.
\end{align*}
In the same way as in Step 1, it is easy to see that $X(t)=Y(t), t\in[0,T], \mathbb{P}-a.s.$ We complete the proof of uniqueness.
\end{proof}

\section{Propagation of Chaos}
Since \eqref{pp} is distribution-dependent, we exploit stochastic interacting particle systems to approximate it. We use a system of $N$ interacting particles,
\begin{equation}\label{interactingparticlesystem}
{\rm d}\big(X^{i,N}_{t}-D(X_{t-\tau}^{i,N})\big)=b(X^{i,N}_{t},X^{i,N}_{t-\tau},\mu^{X,N}_{t}){\rm d}t+(\mu^{X,N}_{t}){\rm d}B^{H,i}_t,\ \ i=1,2,\cdots,N,
\end{equation}
with the initial $X_{0}^{i,N}=\xi^i\in L^{\bar{p}}(\Omega\to\mathbb{R}^d,\mathcal{F}_{0},\mathbb{P})$, where the empirical measures is defined by
\begin{equation*}
\mu^{X,N}_{t}(\cdot)=\frac
{1}{N}\sum_{j=1}^N\delta_{X_{t}^{j,N}}(\cdot) ,
\end{equation*}
here, $\delta_{x}$ denotes the Dirac measure at point $x$. We quote Lemma 3.2 in \cite{HGZ24} as the following Lemma.

\begin{lem}\label{wassersteindistancechazhi}
{\rm For two empirical measures $\mu_{t}^N=\frac{1}{N}\sum\limits_{j=1}\limits^{N}\delta_{Z^{1,j,N}_{t}}$ and $\nu_{t}^N=\frac{1}{N}\sum\limits_{j=1}\limits^{N}\delta_{Z_{t}^{2,j,N}}$ , then
\begin{align*}
\mathbb{E}\left(\mathbb{W}_{\theta}(\mu_{t}^N,\nu_{t}^N)\right)\leq \mathbb{E}\bigg(\frac{1}{N}\sum_{j=1}^N\bigg\lvert Z_{t}^{1,j,N}-Z_{t}^{2,j,N}\bigg\rvert^{\theta}\bigg)^{1/\theta}.
\end{align*}
}
\end{lem}

In order to prove that the particle approximation is useful, we provide the path propagation of the chaos, and consider a system of non-interacting particles as follows:
\begin{equation}\label{noninteractingparticlesystem}
{\rm d}\big(X_{t}^i-D(X_{t-\tau}^i)\big)=b(X_{t}^i,X_{t-\tau}^i,\mathcal{L}_{X_{t}^i}){\rm d}t+\sigma(\mathcal{L}_{X_{t}^i}){\rm d}B^{H,i}_{t},\ \
 X_{0}^i=X_{0}^{i,N}=\xi^i,\ \ i=1,2,\cdots,N.
\end{equation}
We remark that particles $\{X_{t}^{i}, i=1,2,\cdots,N\}$ are mutually independent whereas particles $\{X_{t}^{i,N}, i=1,2,\cdots,N\}$ are not independent but identically distributed.

\begin{lem}\label{Xityoujie}
{\rm Let Assumption \ref{duoxiangshizengz} hold. For ${\bar{p}}>1/H$, we have
\begin{align*}
 \mathbb{E}\bigg(\sup_{t\in[-\tau,T]}|X_{t}^{i}|^{\bar{p}}\bigg)+\mathbb{E}\bigg(\sup_{t\in[-\tau,T]}|X_{t}^{i,N}|^{\bar{p}}\bigg)\leq C,\ \ i=1,2,\cdots,N,
\end{align*}
where $C$ is a positive constant depending on ${\bar{p}},T,L,H, \xi$. }
\end{lem}

\begin{proof}
For any $t\in[0,T]$, by Lemma \ref{Maodebudengshi}, we get
\begin{align*}
\sup_{s\in[0,t]}|X_{s}^{i,N}|^{\bar{p}}\leq \frac{\lambda}{1-\lambda}\|\xi\|^{\bar{p}}+\frac{1}{1-\lambda}\sup_{s\in[0,t]}|X_{s}^{i,N}-D(X_{s-\tau}^{i,N})|^{\bar{p}}.
\end{align*}
Referring to \eqref{2.14}, by the elementary inequality, the H\"{o}lder inequality and Lemma \ref{xishuzengzhangtiaojian}, we get
\begin{align*}
&\mathbb{E}\bigg(\sup_{s\in[0,t]}|X_{s}^{i,N}-D(X_{s-\tau}^{i,N})|^{\bar{p}}\bigg)\\
\leq &3^{{\bar{p}}-1}\mathbb{E}|\xi^i(0)-D(\xi^i(-\tau))|^{\bar{p}}+3^{{\bar{p}}-1}\mathbb{E}\bigg(\sup_{s\in[0,t]}\bigg\lvert \int_{0}^sb(X_{r}^{i,N},X_{r-\tau}^{i,N},\mu_{r}^{X,N}) {\rm d}r\bigg\rvert^{\bar{p}}\bigg)\\
&+3^{{\bar{p}}-1}\mathbb{E}\bigg(\sup_{s\in[0,t]}\bigg\lvert \int_{0}^s\sigma(\mu_{r}^{X,N}){\rm d}B_{r}^{H,i}\bigg\rvert ^{\bar{p}}\bigg)\\
\leq &C+(3t)^{{\bar{p}}-1}\mathbb{E}\int_{0}^tL^{\bar{p}}\bigg[1+|X_{r}^{i,N}|^{\bar{p}}+|X_{r-\tau}^{i,N}|^{(l+1){\bar{p}}}+(\mathbb{W}_{\theta}(\mu_{r}^{X,N},\delta_{0}))^{\bar{p}}\bigg]{\rm d}r\\
&+3^{{\bar{p}}-1}C\int_{0}^t\parallel\sigma(\mu_{r}^{X,N})\parallel^{\bar{p}}{\rm d}r\\
\leq &C+C\mathbb{E}\int_{0}^t|X_{r}^{i,N}|^{\bar{p}}{\rm d}r+C\mathbb{E}\int_{0}^t|X_{r-\tau}^{i,N}|^{(l+1){\bar{p}}}{\rm d}r+C\int_{0}^t(\mathbb{W}_{\theta}(\mu_{r}^{X,N},\delta_{0}))^{\bar{p}}{\rm d}r\\
\leq &C+C\mathbb{E}\int_{0}^t|X_{r}^{i,N}|^{\bar{p}}{\rm d}r+C\mathbb{E}\int_{0}^t|X_{r-\tau}^{i,N}|^{(l+1){\bar{p}}}{\rm d}r+C\mathbb{E}\int_{0}^t\bigg(\frac{1}{N}\sum_{j=1}^N|X_{r}^{j,N}|^{\theta}\bigg)^{{\bar{p}}/\theta}{\rm d}r,
\end{align*}
where for the last inequality we have used the fact that
\begin{align*}
\mathbb{W}_{\theta}(\mu_{s}^{X,N},\delta_{0})=\mathcal{W}_{\theta}(\mu_{s}^{X,N})=\bigg(\frac{1}{N}\sum_{j=1}^N|X_{s}^{j,N}|^{\theta}\bigg)^{1/\theta}.
\end{align*}
Since all $j$ are identically distributed, by the Minkowski inequality and Lemma \ref{wassersteindistancechazhi}, for the last term, we get
\begin{align*}
\mathbb{E}\bigg(\bigg(\frac{1}{N}\sum_{j=1}^N|X_{s}^{j,N}|^{\theta}\bigg)^{{\bar{p}}/\theta}\bigg)
\leq& \bigg(\frac{1}{N}\sum_{j=1}^N\parallel |X_{s}^{j,N}|^{\theta}\parallel_{L^{{\bar{p}}/\theta}}\bigg)^{{\bar{p}}/\theta}\\
=&\bigg(\frac{1}{N} \sum_{j=1}^N(\mathbb{E}|X_{s}^{j,N}|^{\bar{p}})^{\theta/{\bar{p}}}\bigg)^{{\bar{p}}/\theta}=\mathbb{E}|X_{s}^{i,N}|^{\bar{p}}.
\end{align*}
Thus,
\begin{align*}
\mathbb{E}|X_{t}^{i,N}|^{\bar{p}}
\leq &C+C\mathbb{E}\int_{0}^t|X_{s}^{i,N}|^{\bar{p}}{\rm d}s+C\mathbb{E}\int_{0}^t|X_{s-\tau}^{i,N}|^{(l+1){\bar{p}}}{\rm d}s.
\end{align*}
By the same technique as in Lemma \ref{boundedness of pp}, we can show that
\begin{align*}
\mathbb{E}\bigg(\sup_{t\in[0,T]}|X_{t}^{i,N}|^{\bar{p}}\bigg)\leq C.
\end{align*}
Similarly, we obtain
\begin{align*}
\mathbb{E}\bigg(\sup_{t\in[0,T]}|X_{t}^{i}|^{\bar{p}}\bigg)\leq C.
\end{align*}
This completes the proof.
\end{proof}

\begin{lem}\label{Hdayuerfenzhiyipropagation}
{\rm Let Assumption \ref{duoxiangshizengz} hold. For $p>1/H$ and $(l+1)p\le\bar{p}$, we have
\begin{equation*}
\mathbb{E}\bigg(\sup_{t\in[0,T]}|X_{t}^i-X_{t}^{i,N}|^p\bigg)\leq C\times\left\{\begin{array}{ll}
   N^{-\frac{1}{2}},\ \ p>\frac{d}{2},\\
     N^{-\frac{1}{2}}\log(1+N),\ p=\frac{d}{2}, \\
     N^{-\frac{p}{d}},\ p\in[2,\frac{d}{2}).
\end{array}
\right.
\end{equation*}
where $C$ is a positive constant depending on $p,T,H,L,\theta$ but independent of $N$.}
\end{lem}

\begin{proof}
Based on the previous discussion, we get
\begin{align*}
  \sup_{s\in[0,t]}|X_{s}^i-X_{s}^{i,N}|^p\leq \frac{1}{(1-\lambda)^p}\sup_{s\in[0,t]}|X_{s}^i-D(X_{s-\tau}^i)-X_{s}^{i,N}+D(X_{s-\tau}^{i,N})|^p.
\end{align*}
For $i=1,2,\cdots,N$, by \eqref{interactingparticlesystem} and \eqref{noninteractingparticlesystem}, we may compute
\begin{align*}
&X_{t}^i-D(X_{t-\tau}^i)-X_{t}^{i,N}+D(X_{t-\tau}^{i,N})\\
=&\int_{0}^t\left[b(X_{s}^i,X_{s-\tau}^i,\mathcal{L}_{X_{s}^i})-b(X_{s}^{i,N},X_{s-\tau}^{i,N},\mu_{s}^{X,N})\right]{\rm d}s+\int_{0}^t\left[\sigma(\mathcal{L}_{X_{s}^i})-\sigma(\mu_{s}^{X,N})\right] {\rm d}B_{s}^{H,i}.
\end{align*}
By the elementary inequality and the H\"{o}lder inequality, Lemma \ref{xishuzengzhangtiaojian}, Lemma \ref{Xityoujie}, and the same technique with \eqref{2.14}, we get
\begin{align*}
&\mathbb{E}\bigg(\sup_{s\in[0,t]}|X_{s}^i-D(X_{s-\tau}^i)-X_{s}^{i,N}+D(X_{s-\tau}^{i,N})|^p\bigg)\\
\leq& 2^{p-1}\mathbb{E}\bigg(\sup_{s\in[0,t]}  \bigg\lvert \int_{0}^s \left[b(X_{r}^i,X_{r-\tau}^i,\mathcal{L}_{X_{r}^i})-b(X_{r}^{i,N},X_{r-\tau}^{i,N},\mu_{r}^{X,N})\right]{\rm d}r\bigg\rvert ^p\bigg)\\
&+2^{p-1}\mathbb{E}\bigg(\sup_{s\in[0,t]}\bigg\lvert \int_{0}^s\left[\sigma(\mathcal{L}_{X_{r}^i})-\sigma(\mu_{r}^{X,N})\right] {\rm d}B_{r}^{H,i}\bigg\rvert^p\bigg)\\
\leq &2^{p-1}t^{p-1}\mathbb{E}\int_{0}^{t}\lvert b(X_{r}^i,X_{r-\tau}^i,\mathcal{L}_{X_{r}^i})-b(X_{r}^{i,N},X_{r-\tau}^{i,N},\mu_{r}^{X,N}) \rvert ^p{\rm d}r\\
&+C\bigg[\int_{0}^{t}\parallel\sigma(\mathcal{L}_{X_{r}^i})-\sigma(\mu_{r}^{X,N})\parallel^{1/H}{\rm d}r\bigg]^{pH}\\
\leq &(2t)^{p-1}\mathbb{E}\int_{0}^{t} L^p\bigg[|X_{r}^i-X_{r}^{i,N}|+(1+|X_{r-\tau}^i|^l+|X_{r-\tau}^{i,N}|^l)|X_{r-\tau}^i-X_{r-\tau}^{i,N}|\\
&+\mathbb{W}_{\theta}(\mathcal{L}_{X_{r}^i},\mu_{r}^{X,N})\bigg]^p{\rm d}r+C\mathbb{E}\int_{0}^{t}(\mathbb{W}_{\theta}(\mathcal{L}_{X_{r}^i},\mu_{r}^{X,N}))^p{\rm d}r\\
\leq &C\mathbb{E}\int_{0}^{t}|X_{r}^i-X_{r}^{i,N}|^p{\rm d}r+C\int_{0}^{t}\mathbb{E}\left(|X_{r-\tau}^i-X_{r-\tau}^{i,N}|^{2p}\right)^{
\frac{1}{2}}{\rm d}r+C\int_{0}^{t}(\mathbb{W}_{\theta}(\mathcal{L}_{X_{r}^i},\mu_{r}^{X,N}))^p{\rm d}r,
\end{align*}
where we have used the fact that $\bar{p}>1/H$. Noticing that
\begin{align*}
\mathbb{W}_{\theta}(\mathcal{L}_{X_{r}^i},\mu_{r}^{X,N})\leq \mathbb{W}_{\theta}(\mathcal{L}_{X_{r}^i},\mu_{r}^{X})+\mathbb{W}_{\theta}(\mu_{r}^{X},\mu_{r}^{X,N}),
\end{align*}
by Lemma \ref{wassersteindistancechazhi} and Lemma \ref{Xityoujie}, we get
\begin{align*}
\mathbb{E}(\mathbb{W}_{\theta}(\mathcal{L}_{X_{r}^i},\mu_{r}^{X,N}))^p= &\mathbb{E}[(\mathbb{W}_{\theta}(\mathcal{L}_{X_{r}^i},\mu_{r}^{X,N}))^{\theta}]^{p/{\theta}}\\
\leq &\mathbb{E}\bigg(2^{\theta-1}(\mathbb{W}_{\theta}(\mathcal{L}_{X_{r}^i},\mu_{r}^{X}))^{\theta}+2^{\theta-1}(\mathbb{W}_{\theta}(\mu_{r}^{X},\mu_{r}^{X,N}))^{\theta}\bigg)^{p/\theta}\\
\leq &C\mathbb{E}(\mathbb{W}_{\theta}(\mathcal{L}_{X_{r}^i},\mu_{r}^{X}))^p+C\mathbb{E}(\mathbb{W}_{\theta}(\mu_{r}^{X},\mu_{r}^{X,N}))^p\\
\leq &C\mathbb{E}\bigg(\frac{1}{N}\sum_{j=1}^N|X_{r}^j-X_{r}^{j,N}|^{\theta}\bigg)^{p/\theta}+C\mathbb{E}(\mathbb{W}_{\theta}(\mathcal{L}_{X_{r}^i},\mu_{r}^{X}))^p\\
\leq &C\mathbb{E}|X_{r}^i-X_{r}^{i,N}|^p+C\mathbb{E}(\mathbb{W}_{\theta}(\mathcal{L}_{X_{r}^i},\mu_{r}^{X}))^p.
\end{align*}
Since $\mathbb{W}_{\theta}(\mu,\nu)\leq \mathbb{W}_{p}(\mu,\nu)$ for $p\geq \theta$, we get
\begin{align*}
&\mathbb{E}\bigg(\sup_{s\in[0,t]}|X_{s}^i-X_{s}^{i,N}|^p\bigg)\\
\leq &C\mathbb{E}\int_{0}^{t}|X_{r}^i-X_{r}^{i,N}|^p{\rm d}r+C\int_{0}^{t}\mathbb{E}\left(|X_{r-\tau}^i-X_{r-\tau}^{i,N}|^{2p}\right)^{\frac{1}{2}}{\rm d}r+C \mathbb{E}\int_{0}^{t}(\mathbb{W}_{p}(\mathcal{L}_{X_{r}^i},\mu_{r}^{X}))^p{\rm d}r.
\end{align*}
Applying the Gronwall inequality, we get
\begin{align*}
\mathbb{E}\bigg(\sup_{s\in[0,t]}|X_{s}^i-X_{s}^{i,N}|^p\bigg)\leq  C\int_{0}^{t}\mathbb{E}\left(|X_{r-\tau}^i-X_{r-\tau}^{i,N}|^{2p}\right)^{\frac{1}{2}}{\rm d}r+C \mathbb{E}\int_{0}^{t}(\mathbb{W}_{p}(\mathcal{L}_{X_{s}^i},\mu_{s}^{X}))^p{\rm d}s.
\end{align*}
Since $(\mathbb{W}_{p}(\mathcal{L}_{X_{s}^i},\mu_{s}^{X}))^p$ is controlled by the Wasserstein distance estimate(\cite{FG15}, Theorem 1), we get
\begin{align*}
\mathbb{E}(\mathbb{W}_{p}(\mathcal{L}_{X_{s}^i},\mu_{s}^{X}))^p \leq &C(\mathbb{E}|X_{s}^i|^{\bar{p}})^{p/{\bar{p}}}\\
 &\times\left\{\begin{array}{ll}
       N^{-\frac{1}{2}}+N^{-\frac{{\bar{p}}-p}{{\bar{p}}}},\ \ p>\frac{d}{2}, {\bar{p}}\neq2p,\\
     N^{-\frac{1}{2}}\log(1+N)+N^{-\frac{{\bar{p}}-p}{{\bar{p}}}},\ p=\frac{d}{2}, {\bar{p}}\neq2p,\\
     N^{-\frac{p}{d}}+N^{-\frac{{\bar{p}}-p}{{\bar{p}}}},\ p\in[2,\frac{d}{2}), {\bar{p}}\neq\frac{d}{d-p}.
 \end{array}
 \right.
\end{align*}
Since we have $2p\le(l+1)p\le\bar{p}$, then, by Lemma \ref{Xityoujie}, the Gronwall inequality, and the same techniques as in Lemma \ref{existenceanduniquenessofsolution}, the result can be shown.
\end{proof}

\section{Strong Convergence Rate}
In this section, we consider the EM approximate solution of \eqref{interactingparticlesystem}. Assume that for any time $T>0$, there exist positive constants $M, m\in \mathbb{N}$ such that $\Delta=\frac{\tau}{m}=\frac{T}{M}$, where $\Delta\in (0,1)$ is the step-size. Consequently, for the time grid $t_{k}=t_{0}+k\Delta$, it holds that $t_{k+1}-\tau =t_{k+1-m}$. For $k=-m,\cdots, 0,$ we set $Y_{t_k}=\xi$ and $t_{0}=0$. We now define the discrete EM scheme for \eqref{interactingparticlesystem}
\begin{equation*}
\left\{\begin{array}{ll}
Y_{t_{k}}^{i,N}=\xi^i(t_{k}), \ \ \ k=-m,-m+1,\cdots,0, \\
Y_{t_{k+1}}^{i,N}-D(Y_{t_{k+1-m}}^{i,N})=Y_{t_{k}}^{i,N}-D(Y_{t_{k-m}}^{i,N})+b(Y_{t_{k}}^{i,N},Y_{t_{k-m}}^{i,N},\mu_{t_{k}}^{Y,N})\Delta\\
\ \ \ \ \ \ \ \ \ \ \ \ \ \ \ \ \ \ \ \ \ \ \ \ \ \ \ \ \ +\sigma(\mu_{t_{k}}^{Y,N})\Delta B_{t_{k}}^{H,i}, \ \ k=0,1,\cdots,M,
\end{array}
\right.
\end{equation*}
where the empirical measures $\mu_{t_{k}}^{Y,N}(\cdot):=\frac{1}{N}\sum\limits_{j=1}^{N}\delta_{Y_{t_{k}}^{j,N}}(\cdot)$ and $\Delta B_{t_{k}}^{H,i}=B_{t_{k+1}}^{H,i}-B_{t_{k}}^{H,i}$. Let
\begin{equation*}
\bar{Y}_{t}^{i,N}=\left\{\begin{array}{ll}
\xi^i(t),\ \ -\tau\leq t\leq 0 , \\
\sum\limits_{k=0}^{[T/\Delta]}Y_{t_{k}}^{i,N} {\rm I}_{[k\Delta,(k+1)\Delta)}(t),\ \ t\geq 0,
\end{array}
\right.
\end{equation*}
and the continuous EM solution writes as
\begin{equation}\label{interactingparticlesystemEMcontinuous}
\left\{\begin{array}{ll}
     Y_{t}^{i,N}=\xi^i(t),\ \ -\tau\leq t\leq 0,  \\
     Y_{t}^{i,N}-D(Y_{t-\tau}^{i,N})=Y_{t_{k}}^{i,N}-D(Y_{t_{k-m}}^{i,N})+\int_{t_{k}}^tb(\bar{Y}_{s}^{i,N},\bar{Y}_{s-\tau}^{i,N},\bar{\mu}_{s}^{Y,N}){\rm d}s+\int_{t_{k}}^t\sigma(\bar{\mu}_{s}^{Y,N}){\rm d}B_{s}^{H,i},
     \\ \ \ \ \ \ \ \ \ \ \ \ \ \ \ \ \ \ \ \ \ \ \ \ \ t\geq 0,
\end{array}
\right.
\end{equation}
where $\bar{\mu}_{s}^{Y,N}:=\frac{1}{N}\sum\limits_{j=1}^N\delta_{\bar{Y}_{t}^{j,N}}(\cdot)$. For $t_{k}\in[-\tau,0]$, $Y_{t_{k}}^{i,N}=\xi^i$, we rewrite \eqref{interactingparticlesystemEMcontinuous} as
\begin{align*}
Y_{t}^{i,N}-D(Y_{t-\tau}^{i,N})=\xi^i(0)-D(\xi^i(-\tau))+\int_{0}^tb(\bar{Y}_{s}^{i,N},\bar{Y}_{s-\tau}^{i,N},\bar{\mu}_{s}^{Y,N}){\rm d}s+\int_{0}^t\sigma(\bar{\mu}_{s}^{Y,N}){\rm d}B_{s}^{H,i} ,\ \ t\geq 0.
\end{align*}
We observe $Y_{t_{k}}^{i,N}=\bar{Y}_{t_{k}}^{i,N}$ for $k=-m,-(m+1),\cdots,M$.

\begin{lem}\label{Hdayuerfenzhiyicontinuousyoujie}
 {\rm Let Assumption \ref{pp} hold. Then, for ${\bar{p}}>1/H$, we have
\begin{align*}
\mathbb{E}\bigg(\sup_{t\in[-\tau,T]}|Y_{t}^{i,N}|^{\bar{p}}\bigg)\leq C,
\end{align*}
where $C$ is a positive constant depending on ${\bar{p}}, H, T, \xi,L$.}
\end{lem}

\begin{proof}
As previous, for $t\in[0,T]$, we know that
\begin{align*}
\sup_{s\in[0,t]}|Y_{s}^{i,N}|^{\bar{p}}\leq \frac{\lambda}{1-\lambda}\|\xi^i\|^{\bar{p}}+\frac{1}{(1-\lambda)^p}\sup_{s\in[0,t]}|Y_{s}^{i,N}-D(Y_{s-\tau}^{i,N})|^{\bar{p}}
\end{align*}
By the elementary inequality, the H\"{o}lder inequality, \eqref{interactingparticlesystemEMcontinuous}, Lemma \ref{xishuzengzhangtiaojian}, and using the techniques of \eqref{2.14}, we get
\begin{align*}
&\mathbb{E}\bigg(\sup_{s\in[0,t]}|Y_{s}^{i,N}-D(Y_{s-\tau}^{i,N})|^{\bar{p}}\bigg)\\
\leq &3^{{\bar{p}}-1}\mathbb{E}|\xi^{i}(0)-D(\xi^i(-\tau))|^{\bar{p}}+3^{{\bar{p}}-1}\mathbb{E}\bigg(\sup_{s\in[0,t]}\bigg\lvert \int_{0}^sb(\bar{Y}_{r}^{i,N},\bar{Y}_{r-\tau}^{i,N},\bar{\mu}_{r}^{Y,N}){\rm d}r\bigg\rvert^{\bar{p}}\bigg)\\
&+3^{{\bar{p}}-1}\mathbb{E}\bigg(\sup_{s\in[0,t]}\bigg\lvert \int_{0}^s\sigma(\bar{\mu}_{r}^{Y,N}){\rm d}B_{r}^{H,i}\bigg\rvert^{\bar{p}}\bigg)\\
\leq &C+C\mathbb{E}\int_{0}^t|\bar{Y}_{r}^{i,N}|^{\bar{p}}{\rm d}r+C\mathbb{E}\int_{0}^t|\bar{Y}_{r-\tau}^{i,N}|^{(l+1){\bar{p}}}{\rm d}r\\
\leq &C+C\int_{0}^t\sup_{0\leq s\leq r}\mathbb{E}|Y_{s}^{i,N}|^{\bar{p}}{\rm d}r+C\int_{0}^t\sup_{0\leq s\leq r}\mathbb{E}|Y_{s-\tau}^{i,N}|^{(l+1){\bar{p}}}{\rm d}r.
\end{align*}
 The desired assertion can be archived by the same steps as in Lemma \ref{Xityoujie}.
\end{proof}

\begin{lem}\label{Yjianbar{Y}}
{\rm Let Assumption \ref{duoxiangshizengz} hold. For $p>1/H$ and $(l+1)p\le\bar{p}$, then
\begin{align*}
\mathbb{E}\bigg(\sup_{k\in[-m,M]}\sup_{t\in[t_{k},t_{k+1}]}|Y_{t}^{i,N}-\bar{Y}_{t}^{i,N}|^p\bigg)\leq C\Delta^{pH},
\end{align*}
where $C$ is a positive constant that depends on $H, p, L, \lambda$ but independent of $\Delta$.}
\end{lem}

\begin{proof}
By Lemma \ref{Maodebudengshi}, for any $t\in[t_{k},t_{k+1}]$, we get
\begin{align*}
\sup_{t\in[t_{k},t_{k+1}]}|Y_{t}^{i,N}-\bar{Y}_{t}^{i,N}|^p\leq \frac{1}{(1-\lambda)^p}\sup_{t\in[t_{k},t_{k+1}]}|Y_{t}^{i,N}-D(Y_{t-\tau}^{i,N})-\bar{Y}_{t}^{i,N}+D(\bar{Y}_{t-\tau}^{i,N})|^p
\end{align*}
For $t\in[t_{k},t_{k+1}]$, we may compute
\begin{align*}
&\mathbb{E}\bigg(\sup_{t\in[t_{k},t_{k+1}]}|Y_{t}^{i,N}-D(Y_{t-\tau}^{i,N})-\bar{Y}_{t}^{i,N}+D(\bar{Y}_{t-\tau}^{i,N})|^p\bigg)\\
=&\mathbb{E}\bigg(\sup_{t\in[t_{k},t_{k+1}]}\bigg\lvert \int_{t_{k}}^tb(\bar{Y}_{s}^{i,N},\bar{Y}_{s-\tau}^{i,N},\bar{\mu}_{s}^{Y,N}){\rm d}s+\int_{t_{k}}^t\sigma(\bar{\mu}_{s}^{Y,N}){\rm d}B_{s}^{H,i}\bigg\rvert ^p\bigg)\\
\leq& 2^{p-1}\mathbb{E}\bigg(\sup_{t\in[t_{k},t_{k+1}]}\bigg\lvert  \int_{t_{k}}^tb(\bar{Y}_{s}^{i,N},\bar{Y}_{s-\tau}^{i,N},\bar{\mu}_{s}^{Y,N}){\rm d}s\bigg\rvert^p\bigg)+2^{p-1}\mathbb{E}\bigg(\sup_{t\in[t_{k},t_{k+1}]}\bigg\lvert \int_{t_{k}}^t\sigma(\bar{\mu}_{s}^{Y,N}){\rm d}B_{s}^{H,i}\bigg\rvert ^p\bigg)\\
=&I_{1}+I_{2},
\end{align*}
where
\begin{align*}
& I_{1}=2^{p-1}\mathbb{E}\bigg(\sup_{t\in[t_{k},t_{k+1}]}\bigg\lvert  \int_{t_{k}}^tb(\bar{Y}_{s}^{i,N},\bar{Y}_{s-\tau}^{i,N},\bar{\mu}_{s}^{Y,N}){\rm d}s\bigg\rvert^p\bigg),\\
&I_{2}=2^{p-1}\mathbb{E}\bigg(\sup_{t\in[t_{k},t_{k+1}]}\bigg\lvert \int_{t_{k}}^t\sigma(\bar{\mu}_{s}^{Y,N}){\rm d}B_{s}^{H,i}\bigg\rvert ^p\bigg).
\end{align*}
By Lemma \ref{xishuzengzhangtiaojian} and the H\"{o}lder inequality, we get
\begin{align*}
I_{1}\leq& 2^{p-1}\Delta^{p-1}\mathbb{E}\int_{t_{k}}^{t_{k+1}}|b(\bar{Y}_{s}^{i,N},\bar{Y}_{s-\tau}^{i,N},\bar{\mu}_{s}^{Y,N})|^p{\rm d}s\\
\leq& 2^{p-1}\Delta^{p-1}\mathbb{E}\int_{t_{k}}^{t_{k+1}}L^p(1+|\bar{Y}_{s}^{i,N}|^p+|\bar{Y}_{s-\tau}^{i,N}|^{(l+1)p}+(\mathbb{W}_{\theta}(\bar{\mu}_{s}^{Y,N},\delta_{0}))^p)\\
\leq& CL^p\Delta^p+C\Delta^{p-1}\int_{t_{k}}^{t_{k+1}}\mathbb{E}|\bar{Y}_{s}^{i,N}|^p{\rm d}s+C\Delta^{p-1}\int_{t_{k}}^{t_{k+1}}\mathbb{E}|\bar{Y}_{s-\tau}^{i,N}|^{(l+1)p}{\rm d}s\\
&+C\Delta^{p-1}\int_{t_{k}}^{t_{k+1}}\mathbb{E}(\mathbb{W}_{\theta}(\bar{\mu}_{s}^{Y,N},\delta_{0}))^p{\rm d}s.
\end{align*}
Similar with Lemma \ref{boundedness of pp}, by the H\"{o}lder inequality, the Kahane-Khintchine formula, the Fubini theorem and Lemma \ref{xishuzengzhangtiaojian}, we get
\begin{align*}
I_{2}&\leq \frac{C(\lambda)^{-p}(p-1)^{p-1}}{(p-1-\lambda p)^{p-1}}\Delta^{p-1-\lambda p}\int_{t_{k}}^{t_{k+1}}\bigg(\int_{0}^r(r-s)^{(\lambda-1)/H}\parallel\sigma(\bar{\mu}_{s}^{Y,N})\parallel^{1/H}{\rm d}s\bigg)^{pH}{\rm d}r\\
&\leq C \Delta^{p-1-\lambda p}\Delta^{p(\lambda+H-1)}\int_{t_{k}}^{t_{k+1}}\parallel\sigma(\bar{\mu}_{s}^{Y,N})\parallel^p{\rm d}s\\
&\leq C \Delta^{pH-1}\int_{t_{k}}^{t_{k+1}}\mathbb{E}(1+\mathbb{W}_{\theta}(\bar{\mu}_{s}^{Y,N},\delta_{0}))^p{\rm d}s\\
&\leq C\Delta^{pH}+C\Delta^{pH-1}\int_{t_{k}}^{t_{k+1}}\mathbb{E}(\mathbb{W}_{\theta}(\bar{\mu}_{s}^{Y,N},\delta_{0}))^p{\rm d}s.
\end{align*}
Combining $I_{1}$ and $I_{2}$, and using Lemma \ref{wassersteindistancechazhi}, we get
\begin{align*}
&\mathbb{E}\bigg(\sup_{t\in[t_{k},t_{k+1}]}|Y_{t}^{i,N}-\bar{Y}_{t}^{i,N}|^p\bigg)\\
\leq &CL^p\Delta^p+C\Delta^{p-1}\int_{t_{k}}^{t_{k+1}}\mathbb{E}|\bar{Y}_{s}^{i,N}|^p{\rm d}s+C\Delta^{p-1}\int_{t_{k}}^{t_{k+1}}\mathbb{E}|\bar{Y}_{s-\tau}^{i,N}|^{(l+1)p}{\rm d}s\\
&+C\Delta^{p-1}\int_{t_{k}}^{t_{k+1}}\mathbb{E}|\bar{Y}_{s}^{i,N}|^p{\rm d}s+C\Delta^{pH}+C\Delta^{pH-1}\int_{t_{k}}^{t_{k+1}}\mathbb{E}|\bar{Y}_{s}^{i,N}|^p{\rm d}s\\
\leq &C\Delta^{pH}+C\Delta^{pH-1}\int_{t_{k}}^{t_{k+1}}\mathbb{E}|\bar{Y}_{s}^{i,N}|^p{\rm d}s+C\Delta^{pH-1}\int_{t_{k}}^{t_{k+1}}\mathbb{E}|\bar{Y}_{s-\tau}^{i,N}|^{(l+1)p}{\rm d}s\\
\leq &C\Delta^{pH}+C\Delta^{pH-1}\int_{t_{k}}^{t_{k+1}}\sup_{t_{k}\leq u\leq s}\mathbb{E}|Y_{u}^{i,N}|^p{\rm d}s+C\Delta^{pH-1}\int_{t_{k}}^{t_{k+1}}\sup_{t_{k}\leq u\leq s}\mathbb{E}|Y_{u-\tau}^{i,N}|^{(l+1)p}{\rm d}s.
\end{align*}
Since $(l+1)p\le\bar{p}$, we immediately derive from Lemma \ref{Hdayuerfenzhiyicontinuousyoujie} that
\begin{align*}
\mathbb{E}\bigg(\sup_{t\in[t_{k},t_{k+1}]}|Y_{t}^{i,N}-\bar{Y}_{t}^{i,N}|^p\bigg)
\leq  C\Delta^{pH}.
\end{align*}
This completes the proof.
\end{proof}

\begin{lem}\label{X-Ypjieju}
{\rm Let Assumption \ref{duoxiangshizengz} hold. For $p>1/H$ and $(l+1)p\le\bar{p}$, we have
\begin{align*}
 \mathbb{E}\bigg(\sup_{t\in[-\tau, T]}|X_{t}^{i,N}-Y_{t}^{i,N}|^p\bigg)\leq C\Delta^{pH},
\end{align*}
where $C$ is a positive constant that depends on $H, p, L, \lambda$ but independent of $\Delta$.}
\end{lem}

\begin{proof}
Similarly, by Lemma \ref{Maodebudengshi}, we derive that
\begin{align*}
\sup_{s\in[0,t]}|X_{s}^{i,N}-Y_{s}^{i,N}|^p\leq \frac{1}{(1-\lambda)^p}\sup_{s\in[0,t]}|X_{s}^{i,N}-D(X_{s-\tau}^{i,N})-Y_{s}^{i,N}+D(Y_{s-\tau}^{i,N})|^p.
\end{align*}
 By the elementary inequality, the H\"{o}lder inequality, the stochastic Fubini theorem and the Kahane-Khintchine formula and Assumption \ref{duoxiangshizengz}, we may compute
 \begin{align*}
&\mathbb{E}\bigg(\sup_{s\in[0,t]}|X_{s}^{i,N}-D(X_{s-\tau}^{i,N})-Y_{s}^{i,N}+D(Y_{s-\tau}^{i,N})|^p\bigg)\\
=&\mathbb{E}\bigg(\sup_{s\in[0,t]}\bigg\lvert \int_{0}^s(b(X^{i,N}_{r},X^{i,N}_{r-\tau},\mu^{X,N}_{r})-b(\bar{Y}_{r}^{i,N},\bar{Y}_{r-\tau}^{i,N},\bar{\mu}_{r}^{Y,N})) {\rm d}r+\int_{0}^s (\sigma(\mu^{X,N}_{r})-\sigma(\bar{\mu}_{r}^{Y,N})){\rm d}B^{H,i}_r\bigg\rvert^p\bigg) \\
\leq &2^{p-1}t^{p-1}\int_{0}^t\bigg\lvert b(X^{i,N}_{r},X^{i,N}_{r-\tau},\mu^{X,N}_{r})-b(\bar{Y}_{r}^{i,N},\bar{Y}_{r-\tau}^{i,N},\bar{\mu}_{r}^{Y,N})\bigg\rvert ^p{\rm d}r\\
&+2^{p-1}C\int_{0}^t\parallel\sigma(\mu^{X,N}_{r})-\sigma(\bar{\mu}_{r}^{Y,N})\parallel^p{\rm d}r\\
\leq &C\mathbb{E}\int_{0}^t\bigg[|X^{i,N}_{r}-\bar{Y}_{r}^{i,N}|+(1+|X^{i,N}_{r-\tau}|^l+|\bar{Y}_{r-\tau}^{i,N}|^l)|X^{i,N}_{r-\tau}-\bar{Y}_{r-\tau}^{i,N}|
+\mathbb{W}_{\theta}(\mu^{X,N}_{r},\bar{\mu}_{r}^{Y,N})\bigg]^p{\rm d}r\\
&+C\mathbb{E}\int_{0}^t(\mathbb{W}_{\theta}(\mu^{X,N}_{s},\bar{\mu}_{s}^{Y,N}))^p{\rm d}s.
 \end{align*}
 Noticing that $ X_{r}^{i,N}-\bar{Y}_{r}^{i,N}= X_{r}^{i,N}-Y_{r}^{i,N}+ Y_{r}^{i,N}-\bar{Y}_{r}^{i,N},$
 and
 \begin{align*}
 \mathbb{W}_{\theta}(\mu^{X,N}_{r},\bar{\mu}_{r}^{Y,N})=\bigg(\frac{1}{N}\sum_{j=1}^N|X^{i,N}_{r}-\bar{Y}_{r}^{i,N}|^{\theta}\bigg)^{1/\theta},
 \end{align*}
then, by Lemmas \ref{wassersteindistancechazhi}, \ref{Xityoujie}, \ref{Hdayuerfenzhiyicontinuousyoujie}, \ref{Yjianbar{Y}}, we get
 \begin{align*}
&\mathbb{E}\bigg(\sup_{s\in[0,t]}|X_{s}^{i,N}-Y_{s}^{i,N}|^p\bigg)\\
\leq &C\Delta^{pH}+C\mathbb{E}\int_{0}^t|X^{i,N}_{r}-Y_{r}^{i,N}|^p{\rm d}r+C\int_{0}^t\mathbb{E}\left(|X^{i,N}_{r-\tau}-Y_{r-\tau}^{i,N}|^{2p}\right)^{\frac{1}{2}}{\rm d}r.
 \end{align*}
Applying the Gronwall inequality and using the techniques as in Theorem \ref{existenceanduniquenessofsolution}, we get
\begin{align*}
&\mathbb{E}\bigg(\sup_{t\in[0,T]}|X_{t}^{i,N}-Y_{t}^{i,N}|^p\bigg)
\leq C\Delta^{pH}.
\end{align*}
This completes the proof.
\end{proof}

\begin{thm}
{\rm Let Assumption \ref{duoxiangshizengz} hold. For $p>1/H$ and $(l+1)p\le\bar{p}$, it holds that
\begin{align*}
\mathbb{E}\bigg(\sup_{t\in[0,T]}|X_{t}^i-Y_{t}^{i.N}|^p\bigg)\leq C\times\left\{
\begin{array}{ll}
 N^{-\frac{1}{2}}+\Delta^{p H},\ \ p>\frac{d}{2},\\
     N^{-\frac{1}{2}}\log(1+N)+\Delta^{p H},\ p=\frac{d}{2}, \\
     N^{-\frac{p}{d}}+\Delta^{p H},\ p\in[2,\frac{d}{2}),
\end{array}
\right.
\end{align*}
where $C$ is a constant independent of $N$ and $\Delta$.}
\end{thm}

\begin{proof}
    By Lemmas \ref{Hdayuerfenzhiyipropagation} and \ref{X-Ypjieju}, the result is obvious.
\end{proof}

\begin{rem}
{\rm In this paper, we only studied the fractional Brownian motion with Hurst exponent $H\in(1/2,1)$. As it mentioned in \cite{HGZ24} and \cite{FHS22}, if $H\in(0,1/2)$ is considered, we can assume that $\sigma(\mu)$ is independent of $\mu$ to get similar results including the wellposeness, the propagation of Chaos, the numerical schemes etc.}
\end{rem}

\section{A Numerical Example}
In order to verify the previous results, a neutral delayed McKean-Vlasov SDE driven by fractional Brownian motion is introduced in this section. By setting $H=0.8$, we simulate the sample paths in the Matlab work environment, and apply the EM numerical method to approximate the example equation.
\begin{exa}
{\rm Consider the following one-dimensional super-linear McKean-Vlasov SDDE driven by fractional Brownian motion
\begin{equation}\label{anli}
{\rm d}\bigg(X_{t}+\frac{1}{2} X_{t-\tau}\bigg)=\bigg(X_{t}+X_{t-\tau}^2+\int_{\mathbb{R}}(X_{t}-y)\mu_t({\rm d}y)\bigg){\rm d}t+\bigg(\int_{\mathbb{R}}(X_{t}-y)\mu_t({\rm d}y)\bigg){\rm d}B_{t}^H,
\end{equation}
where the initial data $X_{0}$ is a random vector with a standard normal distribution,  $\mu_{t}$ denotes $\mathcal{L}({X_{t}})$ in \eqref{pp}. The corresponding interacting particle system is
\begin{align*}
  {\rm d}\bigg(X_{t}^{i,N}+\frac{1}{2}X_{t-\tau}^{i,N}\bigg)=&\bigg(X_{t}^{i,N}+(X_{t-\tau}^{i,N})^2+\frac{1}{N}\sum_{j=1}^N(X_{t}^{i,N}-X_{t}^{j,N})\bigg){\rm d}t\\
  &+\bigg(\frac{1}{N}\sum_{j=1}^N(X_{t}^{i,N}-X_{t}^{j,N})\bigg){\rm d}B_{t}^{H,i},
\end{align*}
where $i=1,2,\cdots, N. $

The strong convergence rate of EM scheme is shown in Figure \ref{figure1 } with $H=0.8$. For both numerical simulations with six step size ($\Delta=2^{-4},2^{-5},2^{-6},2^{-7},2^{-8},2^{-9}$), we use 5000 sample paths to calculate the mean square error at $T=1$, and use $\Delta=2^{-12}$ to approximate the exact solution. The solid blue line is a reference line with a slope of 1. The observations verify the previous theoretical results.}
\begin{figure}[H]
    \centering
    \includegraphics{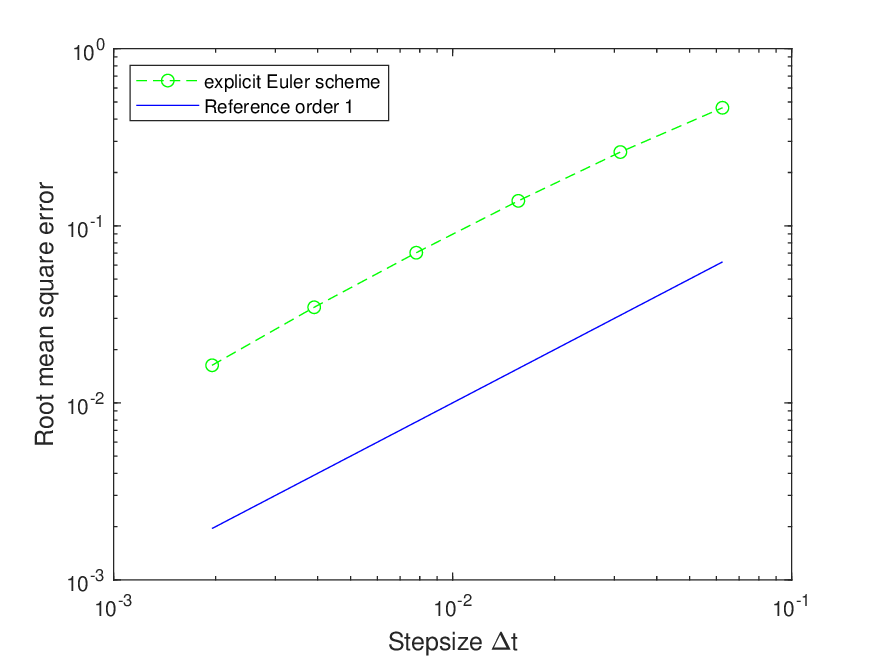}
    \caption{  Mean square error of the EM method for equation \eqref{anli} with $H=0.8$}
    \label{figure1 }
\end{figure}
\end{exa}


\begin{thebibliography}{50}

\bibitem{by13} Bao J., Yuan C., Convergence rate of EM scheme for SDDEs, P. AM. Math. Soc., 141(9): 3231-3243, 2013.

\bibitem{BM19} Bauer M., Meyer-Brandis T., McKean-Vlasov equations on infinite-dimensional Hilbert spaces with irregular drift and additive fractional noise, arXiv: 1912.07427, 2019.

\bibitem{BT97} Bossy M., Talay D., A stochastic particle method for the McKean-Vlasov and the Burgers equation, Math. Comput., 66(217): 157-192, 1997.

\bibitem{BT96} Bossy M., Talay D., Convergence rate for the approximation of the limit law of weakly interacting particles: application to the Burgers equation, Ann. Appl. Probab., 6(3): 818-861, 1996.

\bibitem{clly22} Cui Y., Li X., Liu Y., et al., Explicit numerical approximations for McKean-Vlasov neutral stochastic differential delay equations, Discrete Cont. Dyn.-A, 16(5): 1111-1141, 2023.

\bibitem{DES22} Dos Reis G., Engelhardt S., Smith G., Simulation of McKean-Vlasov SDEs with super-linear growth, IMA J Numer. Anal., 42(1): 874-922, 2022.

\bibitem{DST19} Dos Reis G., Salkeld W., Tugaut J., Freidlin-Wentzell LDP in path space for McKean-Vlasov equations and the functional iterated logarithm law, Ann. Appl. Probab., 29(3): 1487-1540, 2019.

\bibitem{FHS22} Fan X., Huang X., Suo Y., Yuan C., Distribution dependent SDEs driven by fractional Brownian motions, Stoch. Proc. Appl., 151: 23-67, 2022.

\bibitem{FG15} Fournier N., Guillin A., On the rate of convergence in Wasserstein distance of the empirical measure, Probab. Theory Rel., 162(3): 707-738, 2015.

\bibitem{gghy24} Gao S., Guo Q., Hu J., et al., Convergence rate in $\mathcal{L}^p$ sense of tamed EM scheme for highly nonlinear neutral multiple-delay stochastic McKean-Vlasov equations, J Comput. Appl. Math., 441: 115682, 2024.

\bibitem{HSS21} Hammersley W. R. P., \u{S}i\v{s}ka D., Szpruch {\L}., Weak existence and uniqueness for McKean-Vlasov SDEs with common noise, arXiv: 1908.00955, 2021.

\bibitem{HGZ24} He J., Gao S., Zhan W., Guo Q., An explicit Euler-Maruyama method for McKean-Vlasov SDEs driven by fractional Brownian motion, Commun. Nonlinear Sci., 130: 107763, 2024.

\bibitem{HGZG23} He J., Gao S.,  Zhan W., Guo Q., Truncated Euler-Maruyama method for stochastic differential equations driven by fractional Brownian motion with super-linear drift coefficient, Int. J Comput. Math., 100(12): 2184-2195, 2023.

\bibitem{HP09}Hu Y., Peng S., Backward stochastic differential equation driven by fractional Brownian motion, SIAM J Control Optm., 48(3): 1675-1700, 2009.


\bibitem{L18} Lacker D., On a strong form of propagation of chaos for McKean-Vlasov equations, Electron Commun. Prob., 23(45): 1-11. 2018.

\bibitem{r11} Lee M. K., Kim J. H., A delayed stochastic volatility correction to the constant elasticity of variance model, Acta. Math. Appl. Sin.-E., 32(3): 611-622, 2016.

\bibitem{LHH23} Li M., Hu Y., Huang C., Wang X., Mean square stability of stochastic theta method for stochastic differential equations driven by fractional Brownian motion, J Comput. Appl. Math., 420: 114804, 2023.

\bibitem{BEV68} Mandelbrot B., van Ness J. W. , Fractional Brownian motions, fractional noises and applications, Siam. Rev., 10(4): 422-437, 1968.

\bibitem{M07} Mao X., Stochastic differential equations and applications, Horwood Publishing, Chichester, UK, 2007.

\bibitem{r8} Mao X., Sabanis S., Delay geometric Brownian motion in financial option valuation, Stochastics., 85(2): 295-320, 2013.

\bibitem{MC66} McKean H. P., A class of Markov processes associated with nonlinear parabolic equations, P. Natl. Acad. Sci. USA., 56(6): 1907-1911, 1966.

\bibitem{RS22} Reisinger C., Stockinger W., An adaptive Euler-Maruyama scheme for McKean-Vlasov SDEs with super-linear growth and application to the mean-field FitzHugh-Nagumo model, J Comput. Appl. Math., 400: 113725, 2022.

\bibitem{SXW22} Shen G., Xiang J., Wu J., Averaging principle for distribution dependent stochastic differential equations driven by fractional Brownian motion and standard Brownian motion, J Differ. Equations, 321: 381-414, 2022.

\bibitem{W18} Wang F., Distribution dependent SDEs for Landau type equations, Stoch. Proc. Appl., 128(2): 595-621, 2018.

\bibitem{WDX22} Wang M., Dai X., Xiao A., Optimal convergence rate of $\theta$-Maruyama method for stochastic Volterra integro-differential equations with Riemann-Liouville fractional Brownian motion, Adv. Appl. Math. Mech., 14(1): 202-217, 2022.


\bibitem{Y19}Yaskov P., A maximal inequality for fractional Brownian motions, J Math. Anal. Appl., 472(1): 11-21, 2019.

\bibitem{ZY21} Zhang S., Yuan C., Stochastic differential equations driven by fractional Brownian motion with locally Lipschitz drift and their implicit Euler approximation, P. Roy. Soc. Edinb. A., 151(4): 1278-1304, 2021.



\end{thebibliography}
\end{document}